\documentclass[12pt]{article}
\usepackage{amssymb, amsmath, graphicx, amsthm}
\usepackage[paper=a4paper,left=25mm,right=25mm,top=25mm,bottom=25mm]{geometry}

\numberwithin{equation}{section}

\usepackage{mathtext}

\usepackage{amscd}
\usepackage{bbm}
\usepackage{mathtools}
\usepackage{enumitem}

\usepackage{comment}

\usepackage{color,xcolor}

\usepackage[noadjust]{cite}

\allowdisplaybreaks[1]

\newtheorem{Theorem}{Theorem}[section]
\newtheorem{Lemma}[Theorem]{Lemma}

\newtheorem{Remark}[Theorem]{Remark}
\newtheorem{Proposition}[Theorem]{Proposition}

\newtheorem*{Remark*}{Remark}

\newtheorem{Assumption}{Assumption}

\newtheorem{Definition}[Theorem]{Definition}

\newcommand{\cB}{\mathcal B}
\newcommand{\cD}{\mathcal D}

\newcommand{\cL}{\mathcal L}

\newcommand{\cU}{\mathcal U}
\newcommand{\cX}{\mathcal X}

\newcommand{\E}{\mathbb E}

\newcommand{\N}{\mathbb N}
\newcommand{\fP}{\mathbb P}

\newcommand{\R}{\mathbb R}

\AtBeginDocument{
  \let\div\relax
  \DeclareMathOperator{\div}{div}
}

\begin{document}

\date{\today}
\renewcommand{\appendixname}{Appendix}

\title{Optimal Control of McKean-Vlasov equations with controlled stochasticity}
\author{Luca Di Persio \and Peter Kuchling}

\maketitle

\abstract{In this article, we analyse the existence of an optimal feedback controller of stochastic optimal control problems governed by SDEs which have the control in the diffusion part. To this end, we consider the underlying Fokker-Planck equation to transform the stochastic optimal control problem into a deterministic problem with open-loop controller.}

\section{Introduction}
The present paper aims at considering the optimal control problem
\begin{equation}\label{oc_sde}
 \text{minimize }\E\Big[\int_0^T g(X(t))+h(u(t,X(t)))dt\Big]+\E g_0(X(T))
\end{equation}
for some functions $g,h$ and $g_0$, subject to either
\begin{equation}\label{linear_sde}
\begin{split}
  dX(t)&=f(X(t))dt+\sqrt{u(t,X(t))}\sigma(X(t))dW(t)
  \\
  X(0)&=x_0
\end{split}
\end{equation}
where $\sigma=((2a_{ij})_{i,j}^d)^\frac{1}{2}$, or by the McKean-Vlasov SDEs
\begin{equation}\label{mckeanvlasov_sde}
\begin{split}
    dX(t)&=b\Bigg(\frac{d\cL_{X(t)}}{dx}(X(t))\Bigg)D(X(t))dt+\Bigg(\frac{2u(t,X(t))\beta\big(\frac{d\cL_{X(t)}}{dx}(X(t))\big)}{\frac{d\cL_{X(t)}}{dx}(X(t))}\Bigg)^{\frac{1}{2}}dW(t)
    \\
    X(0)&=X_0.
\end{split}
\end{equation}
In both situations, the controller $u$ appears in the stochasticity part and is taken from a set
\begin{displaymath}
    \cU=\{u\in L^\infty((0,T)\times\R^d)\colon\gamma_1\leq u\leq\gamma_2, \sum_{i,j=1}^dD_{ij}^2(a_{ij}(x)u(t,x))\leq\gamma_3\text{ a.e. }(t,x)\in(0,T)\times\R^d\}.
\end{displaymath}
where $(a_{ij})$ is an elliptic matrix for all $x\in\R^d$ in the case \eqref{linear_sde}, while it is represented by $a_{ij}=\delta_{ij}$ for \eqref{mckeanvlasov_sde}. The assumptions on all the other coefficients will be later specified. 
Functions $u\in\cU$ are also called quasi-concave.

Let us recall that the above equations have been analysed in various contexts related to mean-field games \cite{MR2295621,MR3305653} as well as statistical Physics as in \cite{FRANK2001455,10.1214/18-ECP150}. In general, the McKean-Vlasov equations have been introduced by the seminal work of McKean \cite{mckean_seminal}. The specific form of marginal dependence of the coefficients is also known as \emph{Nemytskii-type}, which was recently studied in works such as \cite{doi:10.1137/21M145481X,10.1214/23-ECP519}.

The specific shape of the volatility term in \eqref{mckeanvlasov_sde} gives rise to the Nemytskii-type or singular McKean-Vlasov equation. Such equations appear as microscopic descriptions of various physical processes, such as the Bose-Einstein statistic (with $\beta(r)=a\log(1+r)$). Solutions of such equations were recently studied in \cite{BARBU2023109980} and the references therein, as well as \cite{doi:10.1137/17M1162780,barbu_roeck_2021,barbu_roeck_2021_corr} in the context of the corresponding SDE and \cite{barbu_annals, barbu_jfa, trevisan}.

Let us put our optimal control approach into context. Other works, e.g., \cite{automatica,MR4414709,djete2020mckeanvlasov,MR3739204,MR3501391,MR3631380}, rely on the classical approach, i.e., using the Hamilton-Jacobi-Bellman (HJB) equation, see also  \cite{MR3674558,MR0454768}. This approach poses certain technical issues. Indeed, in these cases, the solution of the HJB equation turns out to be only a viscosity solution, which is not differentiable and hence makes it difficult to analyse the dynamic programming equation to obtain the optimal controller. Furthermore, the dynamic programming approach needs the technical notion of differentiability concerning measures using the Wasserstein distance.

Recently, a different path has been considered, mainly based on transferring the stochastic optimal control problem with feedback inputs to a deterministic problem, which yields an open-loop controller serving as the solution to Problem \eqref{oc_sde}. Some successful attempts have been made by employing the corresponding Kolmogorov (backward) equation in \cite{anita_stoch,barbu_reflection,barbu_roeckner_zhang}. Also, in \cite{MR1920272}, a direct approach using BSDEs via the Kolmogorov approach yields the existence of an optimal controller in the drift term. However, the existence of an optimal controller follows if the optimal value function, which is the solution of the corresponding dynamic equation, is sufficiently smooth.

An alternative solution is proposed in this work.
 In particular, we consider the Fokker-Planck equation (or forward Kolmogorov equation) associated with the above SDEs. So, we rewrite the problem as an optimal control problem governed by a deterministic parabolic equation (the Fokker-Planck equation). This approach has been used in \cite{anita_fpe} in case the controller appears in the drift of the SDE. It has also been mentioned in \cite{MR3501391} as a possible alternative to the Hamilton-Jacobi-Bellman approach.

Since we wish to model a controller in the diffusive term, the approach used in \cite{anita_fpe} must be slightly modified. Indeed, since it is generally to be avoided to impose too strong regularity assumptions on the controller, we consider the Fokker-Planck equations corresponding to \eqref{linear_sde} or \eqref{mckeanvlasov_sde} in a variational setting using the triplet $L^2\subset H^{-1}\subset(L^2)^\ast$. In both cases, the first goal is to show the existence of a solution $\rho^u$ to the corresponding Fokker-Planck equation for any controller $u\in\cU$. This is done in Proposition \ref{prop:sol_existence} and Theorem \ref{thm:sol_existence_mv}, respectively. Next, we show that the sequence of controllers minimising the cost functional $I$ has a limit $u^\ast\in\cU$ corresponding to its solution $\rho^{u^\ast}$, our optimal controller. This is proven in Theorem \ref{thm:minimizer_linear} and Theorem \ref{thm:ex_control_nonlin}, respectively. A question that will be answered in future work is the representation of the controller via the maximum principle.

While using $H^{-1}$ as pivot space seems counterintuitive initially, this space appears naturally in studying the regularity of solutions whose initial conditions have finite first and second moments. As explained in \cite{BARBU2023109980}, solutions with $\rho_0\in L^1\cap L^2$ are differentiable in $H^{-1}$, while such a statement fails concerning the semigroup approach on $L^1$.

The rest of the article is structured as follows. In Section \ref{sec:linear}, we show the existence of an optimal controller for the linear problem. Section \ref{sec:nonlinear} is devoted to the analogous results in the case of a nonlinear McKean-Vlasov equation. For the sake of completeness, we included, within the Appendix section, some classical results regarding the existence of solutions to our PDEs due to Lions.

\subsubsection*{Notation}

We denote by $L^p(\R^d)$, $1\leq p\leq \infty$, the Banach spaces of $p$-Lebesgue integrable functions on $\R^d$. Furthermore, $H^1(\R^d)$ denotes the Sobolev space $H^1(\R^d)=\{u\in L^2(\R^d)\colon \nabla u\in L^2(\R^d)\}$, and $H^{-1}(\R^d)$ its dual. Denote the corresponding dual spaces by $L^2_\mathrm{loc}(\R^d)$ etc.. Similarly, we write $W^{1,\infty}(\R^d)=\{u\in L^\infty(\R^d)\colon \nabla u\in L^\infty(\R^d)\}$.

By $\cD(\R^d)=C^\infty_c(\R^d)$, we denote all infinitely differentiable functions with compact support. The symbols $C(\R^d)$ and $C_b(\R^d)$ represent the space of continuous functions and the space of bounded continuous functions, respectively. By $C^1(\R^d)$ we denote the space of all continuously differentiable functions. Partial derivatives in the $i$-th variable are denoted by $D_i$.

For a real Banach space $\cX$ and $0<T<\infty$, we denote by $L^p(0,T;\cX)$ the space of Bochner $p$-integrable functions $u\colon(0,T)\to\cX$ and by $C([0,T];\cX)$ the space of $\cX$-valued continuous functions on $[0,T]$.

\section{Stochastic Optimal Control: The SDEs \eqref{linear_sde}}\label{sec:linear}
\subsection{Setup}
We consider the following optimal control problem:
\begin{equation}\tag{\ref{oc_sde}}
 \text{minimize }\E\Big[\int_0^T g(X(t))+h(u(t,X(t)))dt\Big]+\E g_0(X(T))
\end{equation}
subject to SDEs
\begin{align*}
  dX&=f(X)dt+\sqrt{u}\sigma(X)dW
  \\
  X(0)&=x_0
\end{align*}
and $u=u(t,x)\in\cU$, where $\cU\subset L^\infty((0,T)\times\R^d)$ is a closed convex set to be made precise below. Furthermore, we assume that $f\in L^\infty(\R^d,\R^d)$ while $h \geq 0$ is convex and continuous. We assume also that
\begin{equation}\label{integrability_cost}
    g\in C_b(\R^d)\cap L^2(\R^d), g\geq 0\text{ and }g_0\in C_b(\R^d), g_0\geq 0.
\end{equation}
If we set $\rho:=\rho(t,x)$ to be the density of $X$, i.e. for $A\in\cB(\R^d)$,
\begin{displaymath}
 \fP(X(t)\in A)=\int_A\rho(t,x)dx,\ \forall\ t\geq 0,
\end{displaymath}
then $\rho$ solves (in the sense of distributions) the Fokker-Planck equation (see \cite{barbu_annals, barbu_jfa, trevisan})
\begin{align}\label{fpe}
 \begin{cases}
   \displaystyle\frac{d}{dt}\rho(t,x)-\frac{1}{2}\sum_{i,j=1}^d D_{ij}^2\left[u(t,x)a_{ij}(x)\rho(t,x)\right]+\div(f(x)\rho(t,x))=0,\ (t,x)\in(0,T)\times\R^d
  \\
  \rho(0,x)=\rho_0(x),\ x\in\R^d
 \end{cases}
\end{align}
where $x\in\R^d$, $a_{ij}=\sum_{k=1}^d\sigma_{ik}\sigma_{kj}$. Note that in particular, we consider the situation where $\rho(t)\geq0 $ and $\rho(t)\in L^1(\R^d)$ for all $t$, as the function $\rho$ models a probability distribution. While this assumption is superfluous for the existence of a solution to \eqref{fpe}, we need to assume $\rho_0\geq 0$ and $\rho_0\in L^1(\R^d)$ for the existence of an optimal controller. By the results from e.g. \cite{barbu_annals}, we then obtain $\rho(t)\geq 0$ and $\rho(t)\in L^1(\R^d)$ for all $t$.

Assume that $a_{ij}\in C(\R^d)$ and there exists $\gamma_0>0$ such that
\begin{equation}\label{definiteness}
 \sum_{i,j=1}^da_{ij}\xi_i\xi_j\geq\gamma_0|\xi|^2,\ \forall\xi\in\R^d.
\end{equation}
By taking $\gamma_0>0$, we assume here that the volatility is non-degenerate. In other words, the covariance of the Brownian term is positive definite. It should be possible to generalise the equation by approximation to degenerate coefficients, i.e., a positive semidefinite correlation. This will be the subject of future work.

Denote the solution to \eqref{fpe} by $\rho^u$. Then, we may rewrite the optimal control problem as the following minimisation problem:
\begin{equation}\label{oc_fpe}
 \text{minimize }I(u)=\left[\int_0^T\int_{\R^d}[g(x)+h(u(s,x))]\rho^u(s,x)dsdx+\int_{\R^d}g_0(x)\rho^u(T,x)dx\right]
\end{equation}
subject to \eqref{fpe} and $u\in\cU$, where
\begin{displaymath}
    \cU=\{u\in L^\infty((0,T)\times\R^d)\colon\gamma_1\leq u\leq\gamma_2, \sum_{i,j=1}^dD_{ij}^2(a_{ij}u(t,x))\leq\gamma_3\}
\end{displaymath}
and $0<\gamma_1<\gamma_2<\infty$.

\begin{Remark}
 In Problem \eqref{oc_sde}, $u=u(t,X(t))$ is a Markov stochastic feedback control, while in \eqref{oc_fpe} it is an open distributed controller.

 The difference to the optimal control problem of \cite{anita_fpe} consists in the fact that while the author applied the control function to the deterministic part, our control problem considers the control function on the stochastic part. This way, the control appears in the second-order derivative instead of the divergence.
\end{Remark}

\subsection{Well-posedness of the state system \eqref{fpe}}\label{chapter_fpe_existence}

In view of Theorem \ref{thm:lions}, we consider \eqref{fpe} in the variational setting w.r.t. the following spaces:
\begin{align*}
 V=L^2(\R^d), H=H^{-1}(\R^d).
\end{align*}
Thus, we have $V^\ast=(L^2)^\ast=(I+A_0)V$, where the dual is taken with respect to the $H^{-1}$-duality functional
\begin{equation}\label{eq:duality_functional}
 {_{V^\ast}}\langle u,v\rangle_{V}=\int_{\R^d}[(I+A_0)^{-1} u]vdx=((I+A_0)^{-1}u,v)_{L^2(\R^d)}=(u,v)_{H^{-1}(\R^d)}
\end{equation}
for $u\in V^\ast, v\in V$, where $A_0$ is the operator
\begin{displaymath}
 A_0\rho=-\sum_{i,j=1}^dD_i(a_{ij}D_j\rho)\ \forall\rho\in V
\end{displaymath}
considered in the sense of distributions on $\R^d$. We note that by \eqref{definiteness}, it follows that $(I+A_0)$ is an isomorphism of $H^1$ onto $H^{-1}$ and thus, expression \eqref{eq:duality_functional} is also used to define a scalar product on $H^{-1}(\R^d)$.
We have
\begin{displaymath}
 V=L^2(\R^d)\subset H=H^{-1}(\R^d)\subset V^\ast=(L^2(\R^d))^\ast,
\end{displaymath}
with dense and continuous embeddings. The operator $A(t)\colon V\to V^\ast$ is defined for $\rho\in L^2(\R^d)$ as
\begin{displaymath}
 A(t)\rho=-\frac{1}{2}\sum_{i,j=1}^dD_{ij}^2(u(t,x)a_{ij}(x)\rho)+\div(f(x)\rho)\text{ in }\cD^\prime(\R^d).
\end{displaymath}
Therefore, the bilinear form in question is defined as follows:
\begin{align*}
 a(t;\rho,\psi)&={_{V^\ast}}\langle A\rho,\psi\rangle_{V}
 \\
 &=-\frac{1}{2}\sum_{i,j=1}^d\int_{\R^d}(I+A_0)^{-1}D_{ij}^2(ua_{ij}\rho)\psi dx+\int_{\R^d}(I+A_0)^{-1}\div f\rho\cdot\psi dx,\ \rho,\psi\in V.
\end{align*}
We have:
\begin{Proposition}\label{prop:sol_existence}
 Assume that there exist constants $0<\gamma_1<\gamma_2$ and $\gamma_0>0$ such that $\gamma_1\leq u\leq\gamma_2$ and \eqref{definiteness}. Furthermore, assume that $a_{ij}\in W^{1,\infty}(\R^d)$, that is,
 \begin{displaymath}
  \overline{a}:=\max\{\max_{i,j}\|D_ia_{ij}\|_\infty,\|a\|_\infty\}<\infty.
 \end{displaymath}
 Let $f\in L^\infty(\R^d,\R^d)$ and $\rho_0\in H=H^{-1}$. Then there exists a unique $H^{-1}$-weak solution $\rho$ of \eqref{fpe}, that is,
 \begin{displaymath}
  \rho\in L^2(0,T;V)\cap C([0,T];H)\text{ and }\frac{d}{dt}\rho\in L^2(0,T;V^\ast).
 \end{displaymath}
\end{Proposition}

Taking into account Theorem \ref{thm:lions}, we must check the following two conditions:
\begin{itemize}
    \item $|a(t,\rho,\varphi)|\leq M\|\rho\|_V\|\varphi\|_V$
    \item $a(t,\rho,\rho)\geq\alpha\|\rho\|_V^2-C\|\rho\|_H^2$
\end{itemize}
for $\rho,\varphi\in\cD(\R^d)$. Since $\cD(\R^d)\subset V$ is dense, there exists a unique continuous extension of $a(t;\cdot,\cdot)\colon\cD(\R^d)\times\cD(\R^d)\to\R$ which is given by our operator. The next subsections are devoted to each of the estimates.

\subsubsection{Upper bound}\label{sec:upper bound}

This section is devoted to Condition 2 of Theorem \ref{thm:lions}. By density, it suffices to show the upper bound for $\rho,\varphi\in \cD(\R^d)$. To this end, we need to rewrite the differential operator in a more suitable form.

\begin{Lemma}\label{lemma_2.3}
 For all test functions $\psi\in\cD(\R^d)$, we have
 \begin{displaymath}
  \sum_{i,j=1}^da_{ij}D_{ij}^2(I+A_0)^{-1}\psi=-\psi+(I+A_0)^{-1}\psi-\sum_{i,j=1}^d(D_ia_{ij})D_j(I+A_0)^{-1}\psi.
 \end{displaymath}
\end{Lemma}

\begin{proof}
 We have by the product rule (denote $D_{ij}^2=D_iD_j=D_jD_i$)
 \begin{align*}
  \sum_{i,j=1}^da_{ij}D_{ij}^2(I+A_0)^{-1}\psi&=\sum_{i,j=1}^da_{ij}D_{ij}^2(I+A_0)^{-1}\psi+\sum_{i,j=1}^d(D_ia_{ij})D_j(I+A_0)^{-1}\psi
  \\
  &\hspace{20pt}-\sum_{i,j=1}^d(D_ia_{ij})D_j(I+A_0)^{-1}\psi
  \\
  &=\sum_{i,j=1}^dD_i[a_{ij}D_j(I+A_0)^{-1}\psi]-\sum_{i,j=1}^d(D_ia_{ij})D_j(I+A_0)^{-1}\psi
  \\
  &=-A_0(I+A_0)^{-1}\psi-\sum_{i,j=1}^d(D_ia_{ij})D_j(I+A_0)^{-1}\psi
  \\
  &=-(I+A_0)^{-1}\psi-A_0(I+A_0)^{-1}\psi+(I+A_0)^{-1}\psi
  \\
  &\hspace{20pt}-\sum_{i,j=1}^d(D_ia_{ij})D_j(I+A_0)^{-1}\psi
  \\
  &=-\psi+(I+A_0)^{-1}\psi-\sum_{i,j=1}^d(D_ia_{ij})D_j(I+A_0)^{-1}\psi
 \end{align*}
\end{proof}

Note that with this approach, we circumvent the regularity assumption on $u$, as the lemma can easily see.

Let now $\rho,\psi\in\cD(\R^d)\subset V$. Using Lemma \ref{lemma_2.3}, we may proceed in the estimation required for Theorem \ref{thm:lions}:
\begin{align*}
 a(t;\rho,\psi)&={_{V^\ast}}\langle A\rho,\psi\rangle_{V}
 \\
 &=-\frac{1}{2}\sum_{i,j=1}^d\int_{\R^d}(I+A_0)^{-1}D_{ij}^2(ua_{ij}\rho)\psi dx+\int_{\R^d}(I+A_0)^{-1}\div(f\rho)\psi dx
 \\
 &=-\frac{1}{2}\sum_{i,j=1}^d\int_{\R^d}D_{ij}^2(ua_{ij}\rho)(I+A_0)^{-1}\psi dx+\int_{\R^d}(I+A_0)^{-1}\div(f\rho)\psi dx
 \\
 &=-\frac{1}{2}\sum_{i,j=1}^d\int_{\R^d}u\rho a_{ij}D_{ij}^2(I+A_0)^{-1}\psi dx+\int_{\R^d}(I+A_0)^{-1}\div(f\rho)\psi dx
 \\
 &=\frac{1}{2}\int_{\R^d}u\rho\psi dx-\frac{1}{2}\int_{\R^d}u\rho(I+A_0)^{-1}\psi dx
 \\
 &\hspace{20pt}+\frac{1}{2}\int_{\R^d}u\rho\sum_{i,j=1}^d(D_ia_{ij})D_j(I+A_0)^{-1}\psi dx+\int_{\R^d}(I+A_0)^{-1}\div(f\rho)\psi dx
\end{align*}
We may estimate the four terms separately (Recall that $V=L^2(\R^d)$). Note that
\begin{itemize}
    \item Differential operators $D_i\colon H^1(\R^d)\to L^2(\R^d)$ are continuous with $\|D_i\|\leq 1$.
    \item The mapping $(I+A_0)^{-1}\colon(L^2)^\ast\to L^2$ is continuous. Denote its operator norm by $\|(I+A_0)^{-1}\|=\alpha$.
    \item Also, $(I+A_0)^{-1}\colon H^{-1}\to H^1$, i.e. for elements from $H^{-1}$, we have higher regularity.
\end{itemize}
Using these facts, we obtain
\begin{enumerate}
 \item
 \begin{align*}
  \Big|\int_{\R^d}u\rho\psi dx\Big|&\leq\gamma_2\int_{\R^d}|\rho\psi|dx\leq\gamma_2\|\rho\|_V\|\psi\|_V
 \end{align*}
 \item 
 \begin{align*}
  \Big|\int_{\R^d}u\rho(I+A_0)^{-1}\psi dx\Big|\leq\gamma_2\alpha\int_{\R^d}|\rho\psi|dx\leq\gamma_2\alpha\|\rho\|_V\|\psi\|_V
 \end{align*}
 \item
 \begin{align*}
  \Big|\int_{\R^d}u\rho\sum_{i,j=1}^d(D_ia_{ij})D_j(I+A_0)^{-1}\psi dx\Big|&\leq d\overline{a}\gamma_2\int_{\R^d}|\rho\sum_{j=1}^dD_j(I+A_0)^{-1}\psi|dx
  \\
  &\leq d^2\overline{a}\gamma_2\alpha\|\rho\|_V\|\psi\|_V
 \end{align*}
 \item
 \begin{align*}
  \Big|\int_{\R^d}(I+A_0)^{-1}\div(f\rho)\psi dx\Big|&\leq\int_{\R^d}|f\rho\nabla(I+A_0)^{-1}\psi|dx
  \\
  &\leq\|f\|_\infty\int_{\R^d}|\rho\sum_{j=1}^dD_j(I+A_0)^{-1}\psi|dx
  \\
  &\leq\|f\|_\infty\alpha d\|\rho\|_V\|\psi\|_V
 \end{align*}
\end{enumerate}
In total, we have
\begin{displaymath}
 |a(t,\rho,\psi)|\leq\Big(\frac{\gamma_2}{2}+\frac{\gamma_2\alpha}{2}+\frac{d^2\overline{a}\gamma_2\alpha}{2}+\|f\|_\infty d\alpha\Big)\|\rho\|_V\|\psi\|_V
\end{displaymath}
which is the desired upper bound. Using that $\cD(\R^d)\subset V$ is dense, we obtain the upper bound for all $\rho,\varphi\in V$.

\subsubsection{Lower bound}

Let $\rho\in\cD(\R^d)$. Again using Lemma \ref{lemma_2.3}, we obtain the following:
\begin{align*}
 a(t;\rho,\rho)&=-\frac{1}{2}\sum_{i,j=1}^d\int_{\R^d}(I+A_0)^{-1}D_{ij}^2(ua_{ij}\rho)\rho dx+\int_{\R^d}(I+A_0)^{-1}\div(f\rho)\rho dx
 \\
 &=-\frac{1}{2}\sum_{i,j=1}^d\int_{\R^d}ua_{ij}\rho D_{ij}^2(I+A_0)^{-1}\rho dx+\int_{\R^d}(I+A_0)^{-1}\div(f\rho)\rho dx
 \\
 &=\frac{1}{2}\int_{\R^d}u\rho^2 dx-\frac{1}{2}\int_{\R^d}u\rho(I+A_0)^{-1}\rho dx
 \\
 &\hspace{20pt}+\frac{1}{2}\sum_{i,j=1}^d\int_{\R^d}u\rho(D_ia_{ij})D_j(I+A_0)^{-1}\rho dx+\int_{\R^d}(I+A_0)^{-1}\div(f\rho)\rho dx
 \\
 &\geq\frac{\gamma_1}{2}\|\rho\|_V^2-\frac{\gamma_2}{2}\|\rho\|_H^2
 \\
 &\hspace{20pt}+\underbrace{\frac{1}{2}\sum_{i,j=1}^d\int_{\R^d}u\rho(D_ia_{ij})D_j(I+A_0)^{-1}\rho dx+\int_{\R^d}(I+A_0)^{-1}\div(f\rho)\rho dx}_{(\ast)}
\end{align*}

For the remaining term $(\ast)$, we proceed in a similar fashion as \cite{anita_fpe}. Consider the two terms separately:
\begin{enumerate}
    \item 
    \begin{align*}
        \Big|\int_{\R^d}u\rho(D_ia_{ij})D_j(I+A_0)^{-1}\rho dx\Big|&\leq\int_{\R^d}|u\rho(D_ia_{ij})D_j(I+A_0)^{-1}\rho|dx
        \\
        &\leq\gamma_2\overline{a}\int_{\R^d}|\rho D_j(I+A_0)^{-1}\rho|dx
        \\
        &\leq\gamma_2\overline{a}\|\rho\|_{L^2}\|(I+A_0)^{-1}\rho\|_{H^1}
        \\
        &\leq\gamma_2\overline{a}\alpha\|\rho\|_{V}\|\rho\|_H
    \end{align*}
    \item
    \begin{align*}
        \Big|\int_{\R^d}(I+A_0)^{-1}\div(f\rho)\rho dx\Big|&\leq\|(I+A_0)^{-1}\div(f\rho)\|_{H^1}\|\rho\|_{H^{-1}}
        \\
        &\leq\alpha\|f\rho\|_{L^2}\|\rho\|_{H^{-1}}\leq\alpha\|f\|_\infty\|\rho\|_{L^2}\|\rho\|_{H^{-1}}
    \end{align*}
\end{enumerate}
Hence, in total, we have
\begin{displaymath}
 |(\ast)|\leq\underbrace{\Big(\frac{d^2\gamma_2\overline{a}\alpha}{2}+\alpha\|f\|_\infty\Big)}_{=:C_2}\|\rho\|_V\|\rho\|_H\leq\frac{C_2\beta}{2}\|\rho\|_V^2+\frac{C_2}{2\beta}\|\rho\|_H^2
\end{displaymath}
where the last inequality is an application of Young's inequality, which is valid for any $\beta>0$. Turning this inequality around, we obtain
\begin{displaymath}
 (\ast)\geq-\frac{C_2\beta}{2}\|\rho\|_V^2-\frac{C_2}{2\beta}\|\rho\|_H^2
\end{displaymath}
Putting everything together, we are left with
\begin{displaymath}
 a(t;\rho,\rho)\geq\frac{\gamma_1-C_2\beta}{2}\|\rho\|_V^2-\Big(\frac{\gamma_2}{2}+\frac{C_2}{2\beta}\Big)\|\rho\|_H^2.
\end{displaymath}
Choosing $\beta$ small enough such that $\gamma_1-C_2\beta>0$, we obtain the desired estimate for Theorem \ref{thm:lions} for $\rho\in\cD(\R^d)$. Using the density of $\cD(\R^d)\subset V$, the bound also holds for all $\rho\in V$.

\begin{Remark}
The existence result of Proposition \ref{prop:sol_existence} can be shown without the additional requirement on the second derivatives of $u$.
\end{Remark}

\subsection{Existence of an optimal controller}\label{sec_4}

The goal is to show the existence of a minimizer for Problem \eqref{oc_fpe}, i.e.,
\begin{displaymath}
 \text{minimize }I(u)=\left[\int_0^T\int_{\R^d}(g(x)+h(u(s,x))\rho^u(s,x)dsdx+\int_{\R^d}g_0(x)\rho^u(T,x)dx\right]
\end{displaymath}
subject to
\begin{displaymath}
 u\in\cU=\{u\in L^\infty((0,T)\times\R^d)\colon 0<\gamma_1\leq u(t,x)\leq\gamma_2, \sum_{i,j=1}^dD_{ij}^2(a_{ij}u(t,x))\leq\gamma_3\}.
\end{displaymath}
Where the estimate on the second derivative holds in the sense of distributions. Note that under the assumptions of Proposition \ref{prop:sol_existence}, the set $\cU$ is closed. This can be seen as follows:

For $u\in L^\infty((0,T)\times\R^d)$, denote by $T_u$ the operator
\begin{displaymath}
    T_u = \sum_{i,j=1}^dD_{ij}^2(a_{ij}u(t,x))
\end{displaymath}
Hence, for $u\in\cU$, the operator $\gamma_3-T_u$ is non-negative. Now, let $\{u_n\}_n\subset\cU$ with $u_n\xrightarrow{n\to\infty}u\in L^\infty((0,T)\times\R^d)$. We need to show that $\gamma_3-T_u$ is positive, i.e., for any $\psi\in L^1((0,T)\times\R^d)$ with $\psi\geq 0$, we have
\begin{equation}\label{eq:proof_u_closed}
    \gamma_3\psi-T_u\psi\geq 0.
\end{equation}
By assumption, we have that \eqref{eq:proof_u_closed} holds for $u$ replaced by $u_n$ for all $n\in\N$. Due to density, it suffices to show \ref{eq:proof_u_closed} for $\psi\in C_c^\infty((0,T)\times\R^d), \psi\geq 0$. For such functions and arbitrary $\varepsilon>0$, there exists $n_0\in\N$ such that for any $n\geq n_0$,
\begin{align*}
    T_u\psi &=\sum_{i,j=1}^da_{ij}u(t,x)D_{ij}^2\psi
    \\
    &=\sum_{i,j=1}^d a_{ij}(u(t,x)-u_n(t,x))D_{ij}^2\psi + T_{u_n}\psi
    \\
    &\leq\sum_{i,j=1}^d a_{ij}(u(t,x)-u_n(t,x))D_{ij}^2\psi + \gamma_3\psi
    \\
    &\leq d\overline{a}\overline{\psi}\varepsilon +\gamma_3\psi,
\end{align*}
where we set $\overline{\psi}=\max\{\max_{i,j}\|D_{ij}^2\psi\|_\infty,\|\psi\|_\infty\}$. Letting $\varepsilon\to 0$, we see that $\cU$ is closed.

We shall assume that \eqref{integrability_cost},\eqref{definiteness} hold and also that $f\in L^\infty(\R^d)$. As the main theorem of this section, we wish to prove the following statement:
\begin{Theorem}\label{thm:minimizer_linear}
 For any $\rho_0\in L^1(\R^d)\cap L^2(\R^d)$ with $\rho_0\geq 0$, there exists a solution $(u,\rho^u)$ to the optimal control problem \eqref{oc_fpe}.
\end{Theorem}

Due to the assumptions on $g,g_0$ and $h$, there exists $m^\ast\in\R$ s.t.
\begin{displaymath}
 \inf_{u\in\cU}I(u)=m^\ast.
\end{displaymath}
Furthermore, we find a sequence $\{u_k\}_{k\in\N}\subset\cU$ such that
 \begin{equation}\label{sequence_k}
  m^\ast\leq I(u_k)\leq m^\ast+\frac{1}{k}\text{ for all }k\in\N.
 \end{equation}

To pass to the limit in \eqref{sequence_k}, we shall first prove some preliminary results given in lemmas which follow.
\begin{Lemma}\label{lemma:priori_bound}
In addition to the assumptions of Proposition \ref{prop:sol_existence}, assume that $\rho_0\in L^2(\R^d)$ and $u\in\cU$. Then the solution $\rho=\rho^u$ to \eqref{fpe} satisfies
\begin{equation}\label{eq:lemma_claim_1}
 \rho\in L^2(0,T;H^1(\R^d))
\end{equation}
\begin{equation}\label{eq:lemma_claim_2}
 \|\rho\|_{L^\infty(0,T;L^2(\R^d))}^2+\int_0^t\int_{\R^d}|\nabla_x\rho(s,x)|^2 ds dx\leq C\|\rho_0\|_{L^2(\R^d)}^2\ \forall t\in(0,T),
\end{equation}
where $C$ is independent of $u$.
\end{Lemma}
\begin{proof}
 We approximate $u$ in $L^\infty((0,T)\times\R^d)$ by a sequence $\{u_\varepsilon\}_\varepsilon\subset L^\infty(0,T;C^2(\R^d))$ with
 \begin{displaymath}
     \sum_{i,j}D_{ij}^2(a_{ij}u_\varepsilon(t,x))\leq\gamma_3\ \forall (t,x)\in(0,T)\times\R^d
 \end{displaymath}
 and $u_\varepsilon\to u$ weak* in $L^\infty((0,T)\times\R^d)$. For $u_\varepsilon\in L^\infty(0,T;C^2(\R^d))$, we can apply as above Thm. \ref{thm:lions} on the spaces
 \begin{displaymath}
     V=H^1(\R^d),\ H=L^2(\R^d),\ V^\ast=H^{-1}(\R^d).
 \end{displaymath}
 Then it follows that Equation \eqref{fpe} with $u_\varepsilon$ instead of $u$ has a unique solution 
 \begin{displaymath}
   \rho_\varepsilon\in L^2(0,T;H^1(\R^d)) \cap C([0,T],L^2(\R^d))\text{ with }\frac{d}{dt}\rho_\varepsilon\in L^2(0,T;H^{-1}(\R^d))
 \end{displaymath}
 Moreover, as easily seen, for $\varepsilon\to 0$, we have
 \begin{displaymath}
     \rho_\varepsilon\to\rho\text{ strongly in }L^2(0,T;L^2(\R^d))\cap C([0,T];H^{-1}(\R^d)).
 \end{displaymath}
 Taking into account that
 \begin{displaymath}
     \frac{1}{2}\frac{d}{dt}\|\rho_\varepsilon(t)\|_{L^2}^2={_{H^1}}(\rho_\varepsilon(t),\frac{d\rho_\varepsilon}{dt}(t))_{H^{-1}}\text{ a.e. }t\in(0,T)
 \end{displaymath}
 we get by \eqref{fpe}
 \begin{align*}
   \frac{1}{2}\|\rho_\varepsilon(t)\|_{L^2}^2&+\int_0^t\int_{\R^d}\rho_\varepsilon(s,x)\Big(-\frac{1}{2}\sum_{i,j=1}^dD_{ij}^2(u(s,x)\rho_\varepsilon(s,x)a_{ij}(x))
   \\
   &\hspace{120pt}+\div(f(x)\rho_\varepsilon(s,x))\Big)dsdx=\frac{1}{2}\|\rho_0\|_{L^2(\R^d)}^2.
 \end{align*}
 On the other hand,
 \begin{align*}
   \int_{\R^d}\rho_\varepsilon D_{ij}^2(u\rho_\varepsilon a_{ij})dx&=-\int_{\R^d}ua_{ij}D_i\rho_\varepsilon D_j\rho_\varepsilon dx-\int_{\R^d}\rho_\varepsilon D_i\rho_\varepsilon D_j(a_{ij}u)dx
 \end{align*}
 and by \eqref{definiteness},
 \begin{align*}
     \sum_{i,j=1}^d\int_{\R^d}\rho_\varepsilon D_{ij}^2(u\rho_\varepsilon a_{ij})dx&=-\int_{\R^d}u\sum_{i,j=1}^da_{ij}(D_i\rho_\varepsilon)(D_j\rho_\varepsilon)dx-\sum_{i,j=1}^d\int_{\R^d}\rho_\varepsilon D_i\rho_\varepsilon D_j(ua_{ij})dx
     \\
     &\leq-\int_{\R^d}u\gamma_0|\nabla\rho_\varepsilon|^2dx-\frac{1}{2}\sum_{i,j=1}^d\int_{\R^d} D_i(\rho_\varepsilon^2) D_j(ua_{ij})dx
     \\
     &\leq-\gamma_1\gamma_0\int_{\R^d}|\nabla\rho_\varepsilon|^2dx+\frac{1}{2}\sum_{i,j=1}^d\int_{\R^d}\rho_\varepsilon^2 D_{ij}^2(ua_{ij})dx
     \\
     &\leq-\gamma_1\gamma_0\int_{\R^d}|\nabla\rho_\varepsilon|^2dx+\frac{\gamma_3}{2}\int_{\R^d}\rho_\varepsilon^2dx.
 \end{align*}
 We have also for any $\delta>0$,
 \begin{align*}
    -\int_{\R^d}\rho_\varepsilon\div(f\rho_\varepsilon)dx&=\int_{\R^d}\nabla\rho_\varepsilon\cdot f\rho_\varepsilon dx
    \\
    &\leq\frac{\delta}{2}\int_{\R^d}|\nabla\rho_\varepsilon|^2dx+\frac{1}{2\delta}\int_{\R^d}|f|^2\rho_\varepsilon^2 dx
    \\
    &\leq\frac{\delta}{2}\int_{\R^d}|\nabla\rho_\varepsilon|^2dx+\frac{\|f\|_\infty^2}{2\delta}\int_{\R^d}\rho_\varepsilon^2 dx\,,
 \end{align*}
 therefore, above calculations together with Young's inequality, give us
 \begin{align*}
   \frac{1}{2}\|\rho_\varepsilon(t)\|_{L^2}^2&=\frac{1}{2}\int_0^t\int_{\R^d}\rho_\varepsilon(s,x)\sum_{i,j=1}^dD_{ij}^2(u(s,x)\rho_\varepsilon(s,x)a_{ij}(x))dsdx
   \\
   &\hspace{20pt}-\int_0^t\int_{\R^d}\rho_\varepsilon(s,x)\div(f(x)\rho_\varepsilon(s,x))dxds+\frac{1}{2}\|\rho_0\|_{L^2}^2
   \\
   &\leq-\gamma_0\gamma_1\int_0^t\int_{\R^d}|\nabla\rho_\varepsilon|^2dxds+\frac{\gamma_3}{2}\int_0^t\int_{\R^d}\rho_\varepsilon^2dxds
   \\
   &\hspace{20pt}+\frac{\delta}{2}\int_0^t\int_{\R^d}|\nabla\rho_\varepsilon|^2dxds+\frac{\|f\|_\infty}{2\delta}\int_0^t\int_{\R^d}\rho_\varepsilon^2dxds+\frac{1}{2}\|\rho_0\|_{L^2}^2
   \\
   &=\Big(\frac{\delta}{2}-\gamma_0\gamma_1\Big)\int_0^t\int_{\R^d}|\nabla\rho_\varepsilon|^2dxds
   \\
   &\hspace{20pt}+\Big(\frac{\gamma_3}{2}+\frac{\|f\|_\infty}{2\delta}\Big)\int_0^t\|\rho_\varepsilon(s)\|_{L^2}^2ds+\frac{1}{2}\|\rho_0\|_{L^2}^2.
 \end{align*}
Let us choose $\delta>0$ s.t. $\delta-2\gamma_0\gamma_1<0$, then rearranging:
 \begin{align*}
   \|\rho_\varepsilon(t)\|_{L^2}^2+(2\gamma_0\gamma_1-\delta)\int_0^t\int_{\R^d}|\nabla\rho_\varepsilon|^2dxds\leq\|\rho_0\|_{L^2}^2+\Big(\gamma_3+\frac{\|f\|_\infty^2}{\delta}\Big)\int_0^t\|\rho_\varepsilon(s)\|_{L^2}^2ds
 \end{align*}
 and exploiting Gronwall's lemma, we have:
 \begin{displaymath}
     \|\rho_\varepsilon(t)\|_{L^2}^2+(2\gamma_0\gamma_1-\delta)\int_0^t\int_{\R^d}|\nabla\rho_\varepsilon(s,x)|^2dsdx\leq C\|\rho_0\|_{L^2}^2\ \forall t\in(0,T),
 \end{displaymath}
 where $C$ is independent of both $\varepsilon$ and $u$. Letting $\varepsilon\to 0$, we get \eqref{eq:lemma_claim_1}. Using the weak lower semicontinuity of the $L^2$- and $H^1$-norms, we also obtain \eqref{eq:lemma_claim_2} as claimed.
\end{proof}

\begin{Lemma}\label{lemma:subsequence_u}
There exists a subsequence $\{u_{k_r}\}_{r\in\N}$ and $u^\ast\in\cU$ such that for any $R>0$:
 \begin{equation}\label{local_weak_convergence}
  u_{k_r}\xrightarrow{r\to\infty}u^\ast\text{ weakly in }L^2((0,T)\times B_R(0),\R)\text{ and weak* in }L^\infty((0,T)\times\R^d).
 \end{equation}
\end{Lemma}
\begin{proof}
Fix $R>0$.
\begin{enumerate}
 \item $\{u_k\}_{k\in\N}$ is weak*-precompact on $L^2((0,T)\times B_R(0),\R)$: By Alaoglu's theorem, any bounded set in a normed space is weak*-precompact. Since $\{u_k\}_{k\in\N}\subset\cU$, we have
 \begin{displaymath}
  \sup_{k\in\N}\|u_k\|_{L^2((0,T)\times B_R(0))}\leq \sqrt{T\lambda(B_R(0))}\gamma_2
 \end{displaymath}
 and hence, the sequence is weak*-precompact.
 \item $\{u_k\}_{k\in\N}$ is weakly precompact on $L^2((0,T)\times B_R(0),\R)$: Since the weak and weak*-topology are equivalent on $L^2$, this statement follows directly from the first step.
 \item $\{u_k\}_{k\in\N}$ has a weakly convergent subsequence: As $\{u_k\}_{k\in\N}$ is weakly compact, this statement follows by Eberlein-Shmulyan. Since $\cU$ is closed, $u^\ast\in\cU$ is clear.
 \item A posteriori, the limit $u^\ast$ of the sequence will again have $T_{u^\ast}\leq\gamma_3$.
\end{enumerate}
Thus, the statement is shown.
\end{proof}
\begin{Lemma}\label{convergence_statement}
In addition to the assumptions of Proposition \ref{prop:sol_existence}, assume that $\rho_0\in L^2(\R^d)$. Then there exists a subsequence $\{\tilde{u}_s\}_{s\in\N}\subset\{u_k\}_{k\in\N}$ such that
 \begin{displaymath}
     \tilde{u}_s\rho^{\tilde{u}_s}\to u^\ast\rho^{u^\ast}
     \text{ strongly in }L^2(0,T;H^{-1}_\mathrm{loc}(\R^d))\text{ as }s\to\infty.
 \end{displaymath}
 Furthermore, we have
 \begin{displaymath}
     \rho^{\tilde{u}_s}\to \rho^{u^\ast}
     \text{ weakly in }L^2(0,T;L^2_\mathrm{loc}(\R^d)).
 \end{displaymath}
\end{Lemma}

\begin{proof}
By Lemma \ref{lemma:priori_bound}, we have
\begin{equation}\label{rho_h1_est}
    \|\rho^{u_{k_r}}(t)\|_{L^2(\R^d)}^2+\|\rho^{u_{k_r}}\|_{L^2(0,T;H^1(\R^d))}\leq C
\end{equation}
where $\{u_{k_r}\}_{r\in\N}$ is the subsequence from Lemma \ref{lemma:subsequence_u}. Therefore,
\begin{equation}\label{eq:derivative_bound}
 \Big\|\frac{d\rho^{u_{k_r}}}{dt}(t)\Big\|_{H^{-1}}\leq C_1\|\rho^{u_{k_r}}(t)\|_{H^1}+C_2.
\end{equation}
Moreover, by Estimate \eqref{rho_h1_est}, it follows that $\{\rho^{u_{k_r}}\}$ is weakly compact in $L^2(0,T;H^1)$ and $L^2(0,T;L^2)$, hence by Aubin-Lions theorem (see, e.g., \cite[Thm. 1.3.5]{barbu_analysis}) there exists a subsequence $\{\tilde{u}_s\}_{s\in\N}$ and a function $\rho^\ast$ s.t.
\begin{equation}\label{rho_convergence}
    \rho^{\tilde{u}_s}\to\rho^\ast\text{ weakly in }L^2(0,T;H^1)\text{ and strongly in }L^2(0,T;L_\mathrm{loc}^2).
\end{equation}
Furthermore, due to \eqref{rho_h1_est} together with \eqref{eq:derivative_bound}, we also get
\begin{displaymath}
  \frac{d\rho^{u_{k_r}}}{dt}\to\frac{d\rho^\ast}{dt}\text{ weakly in }L^2(0,T;H^{-1}).
\end{displaymath}
Taking into account that $\tilde{u}_s\to u^\ast$ weak-star in $L^\infty$, it follows by \eqref{rho_convergence} that
\begin{displaymath}
    \tilde{u}_s\rho^{\tilde{u}_s}\to u^\ast\rho^\ast\text{ weakly in }L_\mathrm{loc}^2((0,T)\times\R^d).
\end{displaymath}
Here is the argument: Let $\psi\in L^2_\mathrm{loc}$. Then $\rho^\ast\psi\in L^1_\mathrm{loc}$ and so for any $R>0$,
\begin{align*}
  \Big|\int_{B_R}(\tilde{u}_s\rho^{\tilde{u}_s}-u^\ast\rho^\ast)\psi dx\Big|&=\Big|\int_{B_R}(\tilde{u}_s\rho^{\tilde{u}_s}-\tilde{u}_s\rho^\ast+\tilde{u}_s\rho^\ast-u^\ast\rho^\ast)\psi dx\Big|
  \\
  &\leq\int_{B_R}|\tilde{u}_s||(\rho^{\tilde{u}_s}-\rho^\ast)\psi|dx+\int_{B_R}|\tilde{u}_s-u^\ast|\rho^\ast\psi dx
  \\
  &\leq\gamma_2\int_{B_R}|\rho^{\tilde{u}_s}-\rho^\ast||\psi|dx+\int_{B_R}|\tilde{u}_s-u^\ast|\rho^\ast\psi dx
  \\
  &\xrightarrow{s\to\infty}0.
\end{align*}
It is left to argue that $\rho^\ast=\rho^{u^\ast}$. But this follows by letting $s\to\infty$ in \eqref{fpe} where $u=\tilde{u}_s$ and we see that $\rho^\ast=\rho^{u^\ast}$.
\end{proof}

\begin{proof}[Proof of Theorem \ref{thm:minimizer_linear}]
 Set $L^u(t,x)=g(x)+h(u(t,x))$. Let $R>0$ and $\{k_s\}_{s\in\N}$ the sequence corresponding to $\{\tilde{u}_s\}_{s\in\N}$ as in \eqref{sequence_k}. By a similar calculation to \cite[Page 22]{anita_fpe}, we obtain
 \begin{align*}
     m^\ast+\frac{1}{k_{r_s}}&\geq\underbrace{\int_0^T\int_{B(0,R)}L^{\tilde{u}_s}\rho^{u^\ast}dxdt}_{=:I_1(\tilde{u}_s)}+\underbrace{\int_0^T\int_{B(0,R)}L^{\tilde{u}_s}(\rho^{\tilde{u}_s}-\rho^{u^\ast})dxdt}_{=:I_2(\tilde{u}_s)}
     \\
     &\hspace{20pt}+\langle g_0,\rho^{\tilde{u}_s}(T)\rangle_{V^\ast,V}
 \end{align*}
 The convergence of the terminal expression is clear. Hence, we focus on the integral expressions.
 \begin{enumerate}
     \item [(i)] Since $h$ is convex and bounded below, via Fatou, the function $I_1(u)$ is lower semicontinuous in $L^p(\R^d)$ for every $1\leq p\leq\infty$, so
     \begin{displaymath}
         \liminf_{s\to\infty}I_1(\tilde{u}_s)\geq I_1(u^\ast)
     \end{displaymath}
     \item [(ii)] By Lemma \ref{convergence_statement} together with Assumption \eqref{integrability_cost}, we obtain the desired convergence. 
     \begin{displaymath}
         \lim_{s\to\infty}I_2(\tilde{u}_s)=0.
     \end{displaymath}
 \end{enumerate}
 By taking $R\to\infty$, this means that $u^\ast$ is a minimizer as claimed.
\end{proof}

\section{Optimal Control of McKean-Vlasov-equation: The nonlinear diffusion case}\label{sec:nonlinear}

The goal of this section is to prove the existence of an optimal controller for the McKean-Vlasov-Equation, i.e. solve the minimising problem
\begin{displaymath}
 \text{minimise }\E(J(X,u))=\E\Big[\int_0^T g(X)+h(u(X))dt+g_0(X(T))\Big]
\end{displaymath}
subject to
\begin{equation}\label{control_space}
 u\in\cU:=\{u\in L^\infty(\R^d)\colon 0<\gamma_1\leq u\leq\gamma_2,\Delta u(x)\leq\gamma_3\text{ a.e. }(t,x)\in(0,T)\times\R^d\}
\end{equation}
and
\begin{equation}\tag{\ref{mckeanvlasov_sde}}
\begin{split}
    dX(t)&=b\Bigg(\frac{d\cL_{X(t)}}{dx}(X(t))\Bigg)D(X(t))dt+\Bigg(\frac{2u(t,X(t))\beta\big(\frac{d\cL_{X(t)}}{dx}(X(t))\big)}{\frac{d\cL_{X(t)}}{dx}(X(t))}\Bigg)^{\frac{1}{2}}dW(t)
    \\
    X(0)&=X_0
\end{split}
\end{equation}
where $\cL_{X(t)}$ is the law of $X(t)$ and $\frac{d\cL_{X(t)}}{dx}$ is its density. To stay in the spirit of Section \ref{sec_4}, we focus here on a solution in $H^{-1}$.

The McKean-Vlasov equation \eqref{mckeanvlasov_sde} is equivalent to the Fokker-Planck equation
\begin{equation}\label{mckeanvlasov_fpe}
\begin{split}
    \frac{d}{dt}\rho&-\Delta(u\beta(\rho))+\div(D(x)b(\rho)\rho)=0\text{ in }(0,\infty)\times\R^d
    \\
    \rho(0)&=\rho_0\text{ in }\R^d
\end{split}
\end{equation}
where $\rho_0dx=\cL_{X_0}$. As explained in \cite{barbu_annals}, by using \cite[Theorem 2.5]{trevisan}, the existence of a solution to \eqref{mckeanvlasov_fpe} yields a probability measure on the canonical space $C([0,T];\R^d)$ solving the corresponding martingale problem. This implies the existence of a weak solution to \eqref{mckeanvlasov_sde} using e.g. \cite[Theorem 4.5.2]{stroock_varadhan}. The other direction from \eqref{mckeanvlasov_sde} to \eqref{mckeanvlasov_fpe} can be seen using It\^o's formula.

Hence, the optimal control problem reduces to
\begin{equation}\label{eq:oc_nonlinear}
 \text{minimize }I(u)=\int_0^T\int_{\R^d}(g(x)+h(u(x)))\rho(t,x)dtdx+\int_{\R^d}g_0(x)\rho(T,x)dx
\end{equation}
subject to \eqref{control_space} and \eqref{mckeanvlasov_fpe}.

While it is desirable from a modelling perspective to have solutions in $L^1$, the appearance of the controller within the diffusion term forces us to rely on $H^{-1}$-methods for the variational approach in the construction of a solution and controller.

To show the existence of an optimal controller, we proceed as in Section \ref{sec_4}. Namely, we use a nonlinear version of Lions' theorem to show the existence of an $H^{-1}$-solution to \eqref{mckeanvlasov_fpe} for any $u\in\cU$. Then we consider an approximating sequence $\{u_s\}_{s\in\N}$ and show that its limit $u^\ast$ attains the infimum of the cost functional.

Let us now formulate the standing assumptions for this section.

\begin{Assumption}\label{assumption:nonlinear}
We assume the following conditions on the coefficients:
\begin{enumerate}
    \item $\beta\in C^1(\R)$, $0<\alpha_0\leq\beta^\prime(r)$ for all $r\in\R$ and $\beta(0)=0$.
    \item $b\in C(\R), |b|_\infty<\infty$, $D\in L^\infty$.
    \item The function $\beta$ is Lipschitz continuous with constant $|\beta|_\mathrm{Lip}$.
    \item There exists $\alpha_2\geq 0$ such that for all $x,y\in\R$,
    \begin{displaymath}
     |b(x)x-b(y)y|\leq\alpha_2|\beta(x)-\beta(y)|.
 \end{displaymath}
\end{enumerate}
\end{Assumption}
\begin{Remark}
 Actually, the condition $|b|_\infty<\infty$ is redundant: If everything else is assumed, for any $x\neq 0$,
 \begin{align*}
   |b(x)|&=\frac{|b(x)x|}{|x|}=\frac{|b(x)x-b(0)0|}{|x|}\leq\alpha_2\frac{|\beta(x)-\beta(0)|}{|x|}\leq\alpha_2|\beta|_\mathrm{Lip}\frac{|x|}{|x|}=\alpha_2|\beta|_\mathrm{Lip}.
 \end{align*}
 By continuity, this bound also holds for $x=0$ and we get $|b|_\infty\leq\alpha_2|\beta|_\mathrm{Lip}$.
 
 Furthermore, it follows directly that the function $y\mapsto b(y)y$ is Lipschitz continuous with constant $\alpha_2|\beta|_\mathrm{Lip}$.
\end{Remark}

\subsection{Well-posedness of the state system \eqref{mckeanvlasov_fpe}}

To show existence of an $H^{-1}$-solution to \eqref{mckeanvlasov_fpe}, we wish to apply Theorem \ref{lions_nonlinear}, again in the case of $V=L^2, H=H^{-1}, V^\ast=(L^2)^\ast$. We consider the nonlinear operator $A\colon V\to V^\ast$, defined by
\begin{displaymath}
  {_{V^\ast}}(Ay,z)_V=\int_{\R^d}(I-\Delta)^{-1}(-\Delta (u\beta(y)))z\, dx+\int_{\R^d}(I-\Delta)^{-1}\div(D(b(y)y))z\, dx
\end{displaymath}
for $y,z\in V$. This leads to the following existence result:
\begin{Theorem}\label{thm:sol_existence_mv}
Under Assumption \ref{assumption:nonlinear} For any $\rho_0\in H$, there exists a solution $\rho\colon[0,T]\to V^\ast$ to \eqref{mckeanvlasov_fpe} satisfying
\begin{gather*}
    \rho\in C([0,T];H^{-1})\cap L^2(0,T;L^2),\ \frac{d\rho}{dt}\in L^2(0,T;(L^2)^\ast)
    \\
    \frac{d\rho}{dt}(t)+Ay(t)=0\text{ a.e. }t\in(0,T),\ \rho(0)=\rho_0.
\end{gather*}
\end{Theorem}
\begin{proof}
 We show that the conditions of Theorem \ref{lions_nonlinear} for $p=2$ are fulfilled for the operator
 \begin{displaymath}
     A(y)=-\Delta(u\beta(y))+\div(Db(y)y),\ y\in V.
 \end{displaymath}
 We start by showing that $A$ is demicontinuous, i.e., for any sequence $\{y_n\}\subset V=L^2(\R^d)$, with $y_n\to y$ in $V$ and any $z\in V$, we show that
 \begin{displaymath}
     {_{V^\ast}}(A(y_n)-A(y),z)_V\xrightarrow{n\to\infty}0.
 \end{displaymath}
 Hence, let $z\in V$. Integration by parts and Cauchy-Schwarz yields
 \begin{align*}
     |{_{V^\ast}}(Ay_n-Ay,z)_V|&\leq\Big|\int (I-\Delta)^{-1}z(-\Delta(u\beta(y_n)-u\beta(y)))dx\Big|
     \\
     &\hspace{20pt}+\Big|\int (I-\Delta)^{-1}z\div(Db(y_n)y_n-Db(y)y)dx\Big|
     \\
     &\leq\int|\Delta(I-\Delta)^{-1}z||u(\beta(y_n)-\beta(y))|dx
     \\
     &\hspace{20pt}+\int|\nabla(I-\Delta)^{-1}zD(x)(b(y_n)y_n-b(y)y)|dx
     \\
     &\leq\|\Delta(I-\Delta)^{-1}z\|_{L^2}\gamma_2|\beta|_\mathrm{Lip}\|y_n-y\|_{L^2}
     \\
     &\hspace{20pt}+|D|_\infty\|\nabla(I-\Delta)^{-1}z\|_{L^2}\|b(y_n)y_n-b(y)y\|_{L^2}
 \end{align*}
 Since $y\mapsto b(y)y$ is Lipschitz continuous, both terms converge to zero due to the $L^2$-convergence of $\{y_n\}_{n\in\N}$. 
 Hence, the demicontinuity is shown.
 
Let us underline that, for the demicontinuity, the Lipschitz condition on $y\mapsto b(y)y$ is unnecessary. We wish to apply the continuous mapping theorem for the second term, which only holds on finite measure spaces. Note that we have $y_n\to y$ in $L^2$ by assumption, implying that $y_n\to y$ locally in measure; hence, by the continuous mapping theorem, it holds that $b(y_n)\to b(y)$ locally in measure, which also means that $b(y_n)y\to b(y)y$ locally in measure. Since $b(y_n)y\in L^2$ and $|b(y_n)y|\leq|b|_\infty|y|\in L^2$, we obtain by Pratt's theorem (see e.g. \cite[Satz 5.3]{elstrodt}) that
 \begin{displaymath}
     b(y_n)y\to b(y)y\text{ in }L^2.
 \end{displaymath}
 
 \textbf{Quasi-Monotonicity:} We need to show that there exists some $\gamma>0$ such that
 \begin{displaymath}
     {_{V^\ast}}(A(y_1)-A(y_2),y_1-y_2)_V\geq -\gamma|y_1-y_2|_H^2
 \end{displaymath}
 for all $y_1,y_2\in D(A)$. First of all, we have
 \begin{align*}
     {_{V^\ast}}(A(y_1)-A(y_2),y_1-y_2)_V&=\int(I-\Delta)^{-1}(A(y_1)-A(y_2))(y_1-y_2)dx
     \\
     &=\int(I-\Delta)^{-1}(-\Delta(u\beta(y_1))+\Delta(u\beta(y_2))(y_1-y_2)dx
     \\
     &\hspace{10pt}+\int(I-\Delta)^{-1}[\div(Db(y_1)y_1-\div(Db(y_2)y_2](y_1-y_2)dx
     \\
     &=-\int(I-\Delta)^{-1}\Delta[u\beta(y_1)-u\beta(y_2)](y_1-y_2)dx
     \\
     &\hspace{10pt}+\int(I-\Delta)^{-1}[\div(Db(y_1)y_1)-\div(Db(y_2)y_2)](y_1-y_2)dx
     \\
     &=\underbrace{\int u(\beta(y_1)-\beta(y_2))(y_1-y_2)dx}_{\geq 0}
     \\
     &\hspace{10pt}-\int(I-\Delta)^{-1}(u\beta(y_1)-u\beta(y_2))(y_1-y_2)dx
     \\
     &\hspace{10pt}+\int(I-\Delta)^{-1}[\div(Db(y_1)y_1)-\div(Db(y_2)y_2)](y_1-y_2)dx
 \end{align*}
 Since the first term is $\geq 0$, we only need to estimate the other two. For the second term, We have
 \begin{align*}
     -\int&(I-\Delta)^{-1}((u\beta(y_1)-u\beta(y_2))(y_1-y_2)dx
     \\
     &=-\int(u\beta(y_1)-u\beta(y_2))(I-\Delta)^{-1}(y_1-y_2)dx
     \\
     &\geq-\gamma_2\int(\beta(y_1)-\beta(y_2))(I-\Delta)^{-1}(y_1-y_2)dx
     \\
     &\geq-\gamma_2|\beta|_\mathrm{Lip}\int(y_1-y_2)(I-\Delta)^{-1}(y_1-y_2)dx
     \\
     &=-\gamma_2|\beta|_\mathrm{Lip}|y_1-y_2|_{H^{-1}}^2.
 \end{align*}
 The divergence term can be estimated as follows:
 \begin{align*}
     \int&(I-\Delta)^{-1}\div[(Db(y_1)y_1)-(Db(y_2)y_2)](y_1-y_2)dx
     \\
     \\
     &=\|\div\|\big(Db(y_1)y_1-Db(y_2)y_2,(I-\Delta)^{-1}(y_1-y_2)\big)_{L^2}
     \\
     &\leq n\|D\|_\infty\big|\big(b(y_1)y_1-b(y_2)y_2,(I-\Delta)^{-1}(y_1-y_2)\big)_{L^2}\big|
     \\
     &\leq n\|D\|_\infty\alpha_2|\beta|_\mathrm{Lip}|y_1-y_2|_{H^{-1}}^2
 \end{align*}
 Turning the inequality around, we obtain
 \begin{displaymath}
   \int(I-\Delta)^{-1}\div[Db(y_1)y_1-Db(y_2)y_2](y_1-y_2)dx\geq -n\|D\|_\infty\alpha_2|\beta|_\mathrm{Lip}|y_1-y_2|_{H^{-1}}^2.
 \end{displaymath}
 In total, we have
 \begin{displaymath}
   {_{V^\ast}}(A(y_1)-A(y_2),y_1-y_2)_V\geq-|\beta|_\mathrm{Lip}(\gamma_2+n\|D\|_\infty\alpha_2)|y_1-y_2|_H^2.
 \end{displaymath}
 
 \textbf{Inequalities from Theorem \ref{lions_nonlinear}:} It is left to show that the two inequalities stated in the theorem hold.
 \begin{enumerate}
     \item For the first one, we may repeat the calculation for the quasi-monotonicity to obtain
     \begin{displaymath}
       {_{V^\ast}}(Ay,y)_V\geq \alpha_0\gamma_1\|y\|^2_V-\gamma|y|_H^2
     \end{displaymath}
     where $\gamma$ is the constant from the calculation of quasi-monotonicity.
     \item Repeating the calculation used for demicontinuity by replacing $y_n-y$ by $y\in V$, we obtain
     \begin{align*}
         \|Ay\|_\ast&=\sup_{\|z\|_V=1}|(Ay,z)|
         \\
         &\leq\|\Delta(I-\Delta)^{-1}z\|_{L^2}\gamma_2|\beta|_\mathrm{Lip}\|y\|_V
         \\
         &\hspace{20pt}+|D|_\infty\|\nabla(I-\Delta)^{-1}z\|_{L^2}\|b(y)y\|_V
         \\
         &\leq C_2\|y\|_V
     \end{align*}
     for some constant $C_2$, where we used that $\|b(y)y\|_V\leq\alpha_2\|\beta(y)\|_V\leq\alpha_2|\beta|_\mathrm{Lip}\|y\|_V$.
 \end{enumerate}
 Since $u>0$ by assumption, the coercivity condition can be omitted given \eqref{eq:blablatest}. Therefore, all conditions of Theorem \ref{lions_nonlinear} are fulfilled, and we have proven the existence of a solution to \eqref{mckeanvlasov_fpe}.
\end{proof}

\subsection{Existence of an optimal controller}

Now that we have established the existence of a solution to \eqref{mckeanvlasov_fpe}, it remains to show the following theorem:
\begin{Theorem}\label{thm:ex_control_nonlin}
  Under Assumption \ref{assumption:nonlinear}, for any $\rho_0\in L^1(\R^d)\cap L^2(\R^d)$ with $\rho_0\geq 0$, there exists a solution $(u,\rho^u)$ to the optimal control problem \eqref{eq:oc_nonlinear}.
\end{Theorem}

\begin{Remark}
    As above, we refer to \cite{barbu_annals} for the nonnegativity and integrability of the solution $\rho$ to \eqref{mckeanvlasov_fpe} under the conditions of Theorem \ref{thm:ex_control_nonlin}.
\end{Remark}
Due to the assumptions on $g,g_0$ and $h$, we again obtain the existence of $m^\ast\in\R$ such that
\begin{displaymath}
  \inf_{u\in\cU}I(u)=m^\ast
\end{displaymath}
as well as a sequence $\{u_k\}\subset\cU$ such that
\begin{displaymath}
  m^\ast\leq I(u_k)\leq m^\ast+\frac{1}{k}\text{ for all }k\in\N.
\end{displaymath}
Lemma \ref{lemma:subsequence_u} is applicable and we get the existence of a subsequence $\{u_{k_r}\}_{r\in\N}$ as well as $u^\ast\in\cU$ such that
\begin{displaymath}
 u_{k_r}\xrightarrow{r\to\infty}u^\ast\text{ weakly in }L^2(B_R(0),\R)\text{ and weak* in }L^\infty(\R^d).
\end{displaymath}

Similarly to Lemma \ref{lemma:priori_bound}, we start out with the following a priori bound on $\rho$:
\begin{Lemma}\label{lemma:priori_nonlin}
Additionally to Assumption \ref{assumption:nonlinear}, also assume that $\rho_0\in L^2(\R^d)$ and $u\in\cU$. Then the solution $\rho=\rho^u$ to Equation \eqref{mckeanvlasov_fpe} satisfies
\begin{equation}\label{eq:l4_claim1}
    \rho\in L^2(0,T;H^1(\R^d))
\end{equation}
\begin{equation}\label{eq:l4_claim2}
 \|\rho\|_{L^\infty(0,T;L^2(\R^d))}^2+\int_0^t\int_{\R^d}|\nabla_x\rho(s,x)|^2 dx ds\leq C\|\rho_0\|_{L^2(\R^d)}^2\ \forall t\in(0,T),
\end{equation}
where $C$ is independent of $u$.
\end{Lemma}
\begin{proof}
Let $\{u_\varepsilon\}_{\varepsilon}\subset C^2([0,T]\times\R^d)\cap\cU$ be a sequence approximating $u$ as $\varepsilon\to 0$. Then \eqref{mckeanvlasov_fpe} yields a solution $\rho_\varepsilon\in L^2$ as in Lemma \ref{lemma:priori_bound}. Furthermore, due to
\begin{displaymath}
  \frac{1}{2}\frac{d}{dt}\|\rho_\varepsilon(t)\|_{L^2}^2={_{H^1}}\big(\rho_\varepsilon(t),\frac{d\rho_\varepsilon}{dt}(t))_{H^{-1}}
\end{displaymath}
we get by \eqref{mckeanvlasov_fpe}
\begin{align*}
    \frac{1}{2}&\|\rho_\varepsilon(t)\|_{L^2}^2+\int_0^t\int_{\R^d}\rho_\varepsilon(s,x)\Big(-\Delta(u_\varepsilon(s,x)\beta(\rho_\varepsilon(s,x))
    \\
    &\hspace{70pt}+\div(D(x)b(\rho_\varepsilon(s,x))\rho_\varepsilon(s,x))\Big)dxds=\frac{1}{2}\|\rho_0\|_{L^2}^2.
\end{align*}
Denoting $\bar{\alpha}:=\max\{\alpha_0,|\beta|_\mathrm{Lip}\}$, we have
\begin{align*}
  \int_{\R^d}\rho_\varepsilon D_{ii}^2(u_\varepsilon\beta(\rho_\varepsilon))dx&=-\int_{\R^d}D_i(\rho_\varepsilon)D_i(u_\varepsilon\beta(\rho_\varepsilon))dx
  \\
  &=-\int_{\R^d}D_i\rho_\varepsilon(D_iu_\varepsilon)\beta(\rho_\varepsilon)dx-\int_{\R^d}(D_i\rho_\varepsilon)\cdot u_\varepsilon\cdot D_i(\beta(\rho_\varepsilon))dx
  \\
  &=-\int_{\R^d}D_i\rho_\varepsilon(D_iu_\varepsilon)\beta(\rho_\varepsilon)dx-\int_{\R^d}(D_i\rho_\varepsilon)\cdot u_\varepsilon\cdot\beta^\prime(\rho_\varepsilon)(D_i\rho_\varepsilon)dx
  \\
  &\leq-\bar\alpha\int_{\R^d}(D_i\rho_\varepsilon)\cdot\rho_\varepsilon\cdot(D_iu_\varepsilon)dx-\alpha_0\gamma_1\int_{\R^d}(D_i\rho_\varepsilon)^2dx
  \\
  &=-\bar\alpha\int_{\R^d}D_i(\rho_\varepsilon
  ^2)\cdot(D_iu_\varepsilon)dx-\alpha_0\gamma_1\int_{\R^d}(D_i\rho_\varepsilon)^2dx
  \\
  &=\bar\alpha\int_{\R^d}\rho_\varepsilon^2 D_{ii}^2 u_\varepsilon dx-\alpha_0\gamma_1\int_{\R^d}(D_i\rho_\varepsilon)^2dx
  \\
  &\leq\bar\alpha\gamma_3\int_{\R^d}\rho_\varepsilon^2dx-\alpha_0\gamma_1\int_{\R^d}(D_i\rho_\varepsilon)^2dx.
\end{align*}
For the divergence term, we get using Young's inequality
\begin{align*}
    -\int_{\R^d}\rho_\varepsilon\div(Db(\rho_\varepsilon)\rho_\varepsilon)dx&=\int_{\R^d}\nabla\rho_\varepsilon\cdot[Db(\rho_\varepsilon)\rho_\varepsilon]dx
    \\
    &\leq\frac{\delta}{2}\int_{\R^d}|\nabla\rho_\varepsilon|^2dx+\frac{1}{2\delta}\int_{\R^d}D^2(b(\rho_\varepsilon)\rho_\varepsilon)^2dx
    \\
    &\leq\frac{\delta}{2}\int_{\R^d}|\nabla\rho_\varepsilon|^2dx+\frac{\|D\|_\infty^2}{2\delta}\int_{\R^d}(b(\rho_\varepsilon)\rho_\varepsilon)^2dx
    \\
    &\leq\frac{\delta}{2}\int_{\R^d}|\nabla\rho_\varepsilon|^2dx+\frac{|b|_\infty^2\|D\|_\infty^2}{2\delta}\int_{\R^d}\rho_\varepsilon^2dx.
\end{align*}
In total, we get
\begin{align*}
    \frac{1}{2}\|\rho_\varepsilon(t)\|_{L^2}^2-\frac{1}{2}\|\rho_0\|_{L^2}^2&=\sum_{i=1}^d\int_0^t\int_{\R^d}\rho_\varepsilon D_{ii}^2(u_\varepsilon\beta(\rho_\varepsilon))dxds-\int_0^t\int_{\R^d}\rho_\varepsilon\div(Db(\rho_\varepsilon)\rho_\varepsilon)dxds
    \\
    &\leq\sum_{i=1}^d\int_0^t\bar\alpha\gamma_3\|\rho_\varepsilon\|_{L^2}^2ds-\sum_{i=1}^d\alpha_0\gamma_1\int_0^t\int_{\R^d}(D_i\rho_\varepsilon)^2dxds
    \\
    &\hspace{20pt}+\frac{\delta}{2}\int_0^t\int_{\R^d}|\nabla\rho_\varepsilon|^2dxds+\frac{|b|_\infty^2\|D\|_\infty^2}{2\delta}\int_0^t\int_{\R^d}\rho_\varepsilon^2dxds
    \\
    &=\bar\alpha\gamma_3d\int_0^t\|\rho_\varepsilon\|_{L^2}^2ds-\alpha_0\gamma_1\int_0^t\|\nabla\rho_\varepsilon\|_{L^2}^2ds
    \\
    &\hspace{20pt}+\frac{\delta}{2}\int_0^t\|\nabla\rho_\varepsilon\|_{L^2}^2ds+\frac{|b|_\infty^2\|D\|_\infty^2}{2\delta}\int_0^t\|\rho_\varepsilon\|_{L^2}^2ds
\end{align*}
Rearranging, this yields
\begin{align*}
    \|\rho_\varepsilon(t)\|_{L^2}^2+(2\alpha_0\gamma_1-\delta)\int_0^t\|\nabla\rho_\varepsilon\|_{L^2}^2ds\leq\|\rho_0\|_{L^2}^2+\big(2\bar\alpha\gamma_3d+\frac{|b|_\infty\|D\|_\infty}{\delta}\big)\int_0^t\|\rho_\varepsilon\|_{L^2}^2ds.
\end{align*}
As in Lemma \ref{lemma:priori_bound}, we proceed by using Gronwall's inequality and send $\varepsilon\to 0$ to obtain \eqref{eq:l4_claim1} and \eqref{eq:l4_claim2} for arbitrary $u\in\cU$.
\end{proof}
Comparing with Section \ref{sec_4}, we can conclude the proof of existence of a controller as follows:
\begin{proof}[Proof of Theorem \ref{thm:ex_control_nonlin}]
In view of Lemma \ref{lemma:priori_nonlin}, the statement of Lemma \ref{convergence_statement} directly carries over to the nonlinear case. Note that due to the Lipschitz properties of $\beta$ and $y\mapsto b(y)y$, it follows from the convergences appearing in Lemma \ref{convergence_statement} together with \eqref{rho_convergence} that $\rho^\ast=\rho^{u^\ast}$. Therefore, the proof of Theorem \ref{thm:ex_control_nonlin} follows the same arguments as the proof of Theorem \ref{thm:minimizer_linear}.
\end{proof}

\newpage
\appendix

\section{The existence theorems}

\subsection{Lions' theorem: Linear case}
The following theorem, see e.g. \cite[Thm 10.9]{brezis}.
\begin{Theorem}\label{thm:lions}
 Let $V,H$ be Hilbert spaces. Assume that $V\subset H$ dense and continuous such that
 \begin{displaymath}
  V\subset H\subset V^*.
 \end{displaymath}
 Let $T>0$ and for almost all $t\in[0,T]$, we are given a bilinear form $a(t;\rho,\varphi)\colon V\times V\to\R$ satisfying the following properties:
 \begin{enumerate}
  \item For all $\rho,\varphi\in V$, the function $t\mapsto a(t;\rho,\varphi)$ is measurable,
  \item There exists $M>0$ such that
  \begin{equation}\label{bounded_form}
   |a(t;\rho,\varphi)|\leq M\|\rho\|_V\|\varphi\|_V
  \end{equation}
  for almost every $t\in[0,T]$ and all $\rho,\varphi\in V$.
  \item There exist constants $\alpha>0$ and $C\in\R$ such that 
  \begin{displaymath}
   a(t;\rho,\rho)\geq \alpha\|\rho\|_V^2-C\|\rho\|_H^2
  \end{displaymath}
  for almost every $t\in[0,T]$ and all $\rho\in V$.
 \end{enumerate}
 Given $\tilde{f}\in L^2(0,T;V^*)$ and $\rho_0\in H$, there exists a unique function $\rho$ satisfying
 \begin{displaymath}
  \rho\in L^2(0,T;V)\cap C([0,T];H)\text{ and }\frac{d}{dt}\rho\in L^2(0,T;V^*)
 \end{displaymath}
 and the differential equation
 \begin{equation}\label{lions_pde}
  \begin{cases}
   \langle \frac{d\rho}{dt}(t),\varphi\rangle+a(t;\rho(t),\varphi)&=\langle \tilde{f}(t),\varphi\rangle\text{ a.e. }t\in(0,T),\ \forall \varphi\in V
   \\
   \rho(0)&=\rho_0.
  \end{cases}
 \end{equation}
\end{Theorem}

\subsection{Lions' theorem: Nonlinear case}

For reference to the following theorems, see, e.g. \cite[Thm. 4.10]{barbu_nonlinear} or \cite[Thm. 2.1.2, p. 162]{lions_book}. We again assume that there exists a triple
\begin{displaymath}
  V\subset H\subset V^\ast
\end{displaymath}
and we denote the norms on $V$ and $H$ by $\|\cdot\|$ and $|\cdot|$, respectively. The norm on $V^\ast$ is denoted by $\|\cdot\|_\ast$.
\begin{Theorem}\label{lions_nonlinear}
Let $A\colon V\to V^\ast$ be a demicontinuous, coercive and quasi-monotone operator, i.e.,
\begin{displaymath}
    {_{V^\ast}}(Au-Av,u-v)\geq-\gamma|u-v|^2_H\text{ for all }u,v\in V.
\end{displaymath}
Furthermore, assume that $A$ satisfies the conditions
\begin{align}
    (Ay,y)&\geq \omega\|y\|^p-C_1|y|_H^2\text{ for all }y\in V\label{eq:blablatest}
    \\
    \|Ay\|_\ast&\leq C_2(1+\|y\|^{p-1})\text{ for all }y\in V\notag
\end{align}
where $\omega>0$ and $p>1$. Given $y_0\in H$ and $\frac{1}{p}+\frac{1}{q}=1$, there exists a unique absolutely continuous function $y\colon[0,T]\to V^\ast$ that satisfies
\begin{gather*}
    y\in C([0,T];H)\cap L^p(0,T;V),\ \frac{dy}{dt}\in L^2(0,T;V^\ast)
    \\
    \frac{dy}{dt}(t)+Ay(t)=0\text{ a.e. }t\in(0,T),\ y(0)=y_0,
\end{gather*}
\end{Theorem}

We provide an English translation of \cite[Theorem 2.1.2, p. 162]{lions_book} to keep the article as self-contained as possible. For the following theorem, we assume that we have the triple
\begin{displaymath}
  V\subset H\subset V^\prime
\end{displaymath}
where $V$ is a Banach space densely and continuously embedded in the Hilbert space $H$. We identify $H$ with its dual, and let $V^\prime$ be the dual of $V$ (w.r.t. $H$).

\begin{Definition}
  An operator $A\colon V\to V^\prime$ is called monotone if
  \begin{displaymath}
    (A(u)-A(v),u-v)\geq 0\text{ for all }u,v\in V.
  \end{displaymath}
\end{Definition}

\begin{Theorem}
  Let $V,H$ as stated above and $V$ separable. Let $A\colon V\to V^\prime$ with the following properties: For some $1<p<\infty$,
  \begin{itemize}
      \item $A$ is hemicontinuous from $V$ to $V^\prime$, i.e.,
      \begin{displaymath}
        \|A(v)\|_\ast\leq c\|v\|^{p-1}
      \end{displaymath}
      \item $A$ is monotone from $V$ to $V^\prime$
      \item For all $v\in V$, we have
      \begin{equation}\label{eq:lions_1.36}
          (A(v),v)\geq \alpha\|v\|^p,\ \alpha>0
      \end{equation}
  \end{itemize}
  Let $f\in L^{p^\prime}(0,T;V^\prime)$ and $u_0\in H$. Then there exists a unique function $u$ with
  \begin{itemize}
    \item $u\in L^p(0,T;V)$
    \item $u^\prime+A(u)=f$
    \item $u(0)=u_0$.
  \end{itemize}
 It is worth noticing that by exploiting the first and second implication, we have that $u^\prime\in L^{p^\prime}(0,T;V^\prime)$, so the last implication makes sense.

  In applications, the above hypotheses may be too strong. Therefore, it is useful to state the following variant of the theorem: Assume that there exists a seminorm $[v]$ on $V$ such that
  \begin{equation}\label{eq:lions_1.41}
      \text{there exists }\lambda>0\text{ and }\beta>0\text{ such that }[v]+\lambda|v|\geq\beta\|v\|
  \end{equation}
  for all $v\in V$. (Here, $|\cdot|$ and $\|\cdot\|$ denote the norms on $H$ and $V$, respectively.) Also, assume that
  \begin{equation}\label{eq:lions_1.42}
      (A(v),v)\geq\alpha[v]^p.
  \end{equation}
  Then we may replace hypothesis \eqref{eq:lions_1.36} by \eqref{eq:lions_1.41} and \eqref{eq:lions_1.42}, and the conclusions of this theorem stay valid.
\end{Theorem}

\section*{Acknowledgments}
We take this opportunity to thank Prof. Viorel Barbu whose suggestions and advice have been fundamental for the successful realization of this
article.

\renewcommand{\refname}{References}

\bibliographystyle{amsplain}
\bibliography{bibliography}
  
\end{document}